\documentclass[a4paper,12pt,reqno]{amsart}

\usepackage[mathscr]{eucal}
\usepackage{amscd}
\usepackage{amsfonts}
\usepackage{amsmath, amsthm, amssymb}
\usepackage{latexsym}
\usepackage[dvips]{graphics}

\usepackage{lipsum} 
\usepackage{cite}
\usepackage[sc]{mathpazo} 
\usepackage[T1]{fontenc} 
\usepackage{microtype} 

\usepackage[hmarginratio=1:1,top=30mm]{geometry} 
\usepackage{hyperref} 
\usepackage{abstract} 

\usepackage{titlesec} 
\renewcommand\thesection{\Roman{section}} 
\titleformat{\section}[block]{\large\scshape\centering}{\thesection.}{1em}{} 
\renewcommand{\thesection}{\arabic{section}}

\newtheorem{thm}{Theorem}[section]
\newtheorem{lem}{Lemma}[section]
\numberwithin{equation}{section}

\newtheorem{cor}{Corollary}[section]

\newtheorem{con}{Conjecture}
\theoremstyle{plain}

\begin{document}
\thispagestyle{empty}



\title{Lie higher derivations of arbitrary triangular algebras \\     
}


\author{Mohammad Ashraf and Mohammad Afajal Ansari}

\address{Mohammad Ashraf, Department of Mathematics,
Aligarh Muslim University,
Aligarh-202002 India}
\email{\href{mailto:mashraf80@hotmail.com}{mashraf80@hotmail.com}}

\address{Mohammad Afajal Ansari, Department of Mathematics, Aligarh Muslim University,
Aligarh-202002 India}
\email{\href{mailto:afzalgh1786@gmail.com}{afzalgh1786@gmail.com}}

\maketitle



\begin{abstract}
Motivated by the works of Wang [Y. Wang, \textit{Lie (Jordan) derivations of arbitrary triangular algebras,} Aequationes Mathematicae, \textbf{93} (2019),  1221-1229] and Moafian et al. [F. Moafian and H. R. Ebrahimi Vishki,  \textit{Lie higher derivations on triangular algebras revisited,} Filomat, \textbf{30}(12) (2016),  3187-3194.], we shall study Lie higher derivations of arbitrary triangular algebras. In fact, it is shown that every Lie higher derivation on an arbitrary triangular algebra is proper, using the notion of maximal left (right) ring of quotients.
\end{abstract}
\vspace{.3cm}

\thispagestyle{empty}

\section{Introduction}

Let $\mathcal{R}$ be a commutative ring with identity, $\mathcal{A}$ be an algebra over $\mathcal{R}$ and $\mathcal{Z}(\mathcal{A})$ be the center of $\mathcal{A}$. Let $[x,y]=xy-yx$ denote the commutator of elements $x,y\in\mathcal{A}.$ An $\mathcal{R}$-linear mapping $\Delta:\mathcal{A}\rightarrow \mathcal{A}$ is called a \textit{derivation} if $\Delta(xy)=\Delta(x)y+x\Delta(y)$ holds for all $x, y\in \mathcal{A}$. An $\mathcal{R}$-linear mapping $\Delta:\mathcal{A}\rightarrow \mathcal{A}$ is said to be a \textit{Lie derivation} if
$\Delta([x, y]) = [\Delta, y] + [x, \Delta(y)]$ for all $x, y\in \mathcal{A}$  holds for all $x, y\in \mathcal{A}$. The structure of Lie derivations on various rings and algebras has been extensively studied  (see  \cite{M64,Mi73,Ch03,YZ10,dw12}).

\par The study of derivations were further extended to higher derivations (see \cite{FH02,FH02(1),LS12,Qi13,QH10,WX11,XW12} and references therein). Let $\mathbb{N}$ be the set of nonnegative integers and $\mathcal{D}=\{d_n\}_{n\in\mathbb{N}}$ be a family of $\mathcal{R}$-linear mappings $d_n:\mathcal{A}\rightarrow\mathcal{A}$ such that $d_0=id_{\mathcal{A}},$ the identity mapping on $\mathcal{A}.$ Then $\mathcal{D}$ is said to be
\begin{itemize}
\item [$(i)$] a higher derivation if $d_n(xy)=\sum\limits_{i+j=n}d_i(x)d_j(y)$ for all $x,y \in \mathcal{A}$ and  $n \in \mathbb{N}.$
\item [$(ii)$] a Lie higher derivation if
$d_n([x,y])=\sum\limits_{i+j=n}[d_i(x),d_j(y)]$ for all $x,y \in \mathcal{A}$ and  $n \in \mathbb{N}.$
\end{itemize}

\par Note that every higher derivation is Lie higher derivation, however, the converse is not true, in general. If $\mathcal{D}=\{d_n\}_{n\in \mathbb{N}}$ is a higher derivation on $\mathcal{A}$ and $\mathcal{H}=\{h_n\}_{n\in \mathbb{N}}$ is a sequence of $\mathcal{R}$-linear mappings $h_n:\mathcal{A}\rightarrow\mathcal{Z}(\mathcal{A})$ such that $h_n([x,y])=0$ for all $x,y\in \mathcal{A}$ and $n \in \mathbb{N},$ then $\mathcal{D}+\mathcal{H}=\{d_{n}+h_n\}_{n\in \mathbb{N}}$ is a Lie higher derivation on $\mathcal{A},$ which is not necessarily a higher derivation. A Lie higher derivation of the above kind is called proper. The natural problem that one considers in this context is under what conditions a Lie higher derivation on an algebra is proper.

\par Several authors have made important contributions to the related topics (see \cite{XW12,LS12,QH10,MEv16} and the references therein). In the year 2000, Cheung \cite{Ch00} initiated the study of linear mappings on triangular algebras. Furthermore, Cheung \cite{Ch03} obtained some sufficient conditions under which every Lie derivation $\mathcal{L}$ on a triangular algebra $\mathfrak{A}$ can be written as $\mathcal{L}=d+h,$ where $d:\mathfrak{A}\rightarrow\mathfrak{A}$ is a derivation and $h:\mathfrak{A}\rightarrow \mathcal{Z}(\mathfrak{A})$ is a linear mapping such that $h([x,y])=0$ for all $x,y\in\mathfrak{A}.$ This result has been generalized in different ways (see \cite{aj17,YZ10,LS12,Qi13,MEv16,B11,XW12(1)}). Li and Shen \cite{LS12} extended Cheung's result to Lie higher derivations of triangular algebras by proving the following statement. Let $\mathfrak{A}=\mathrm{{Tri}}(\mathcal{A},\mathcal{M},\mathcal{B})$ be a triangular algebra such that $\pi_{\mathcal{A}}(\mathcal{Z}(\mathfrak{A}))= \mathcal{Z}(\mathcal{A})$ and $\pi_{\mathcal{B}}(\mathcal{Z}(\mathfrak{A}))=\mathcal{Z}(\mathcal{B}),$ where $\pi_{\mathcal{A}},\pi_{\mathcal{B}}$ are natural projections of $\mathfrak{A}.$ Then every Lie higher derivation $\mathcal{L}=\{L_n\}_{n\in\mathbb{N}}$ of $\mathfrak{A}$ is proper, that is, for each $n\in\mathbb{N},$ $L_n=d_n+h_n,$ where $\{d_n\}_{n\in\mathbb{N}}$ is a higher derivation on  $\mathfrak{A}$ and $\{h_n\}_{n\in\mathbb{N}}$ is a sequence of linear mappings $h_n:\mathfrak{A}\rightarrow \mathcal{Z}(\mathfrak{A})$ such that $h_n([x,y])=0$ for all $x,y\in\mathfrak{A}$ and $n\in\mathbb{N}.$ Moafian and Ebrahimi Vishki \cite{MEv16} given the structure of Lie higher derivations on triangular algebras by the entries of matrices and obtained a similar conclusion as shown by Li and Shen \cite{LS12}.

\par In the year 1956, Utumi \cite{U56} introduced the concept of the maximal left ring of quotients and proved that every unital ring has a maximal left (right) ring of quotients. For a detailed study of maximal left (right) ring of quotients the reader is referred to \cite{BMM96, BCM07}.  Eremita \cite{E15} explored functional identities of degree $2$ for a more general class of triangular rings using the notion of maximal left ring of quotient.  Wang \cite{W15} considered functional identities of degree $2$ in arbitrary triangular rings. Furthermore, in \cite{W19}, Wang  constructed a triangular algebra from a given triangular algebra, using the notion of maximal left (right) ring of quotients and gave a description of Lie derivations on arbitrary triangular algebras through the constructed triangular algebra. Recently, Ashraf et al. \cite{aaa21} characterized generalized Lie triple derivation on arbitrary triangular algebras.

\par Motivated by the above works, in this paper, we shall study Lie higher derivations on arbitrary triangular algebras which generalizes the result of Wang \cite{W19}. In fact, it is shown that every Lie higher derivation on an arbitrary triangular algebra is proper. It is worth to mention that our characterization of Lie higher derivations on triangular algebras (Theorem \ref{thm1}) is different from that in \cite{LS12,MEv16}.

\section{Preliminaries}
Let $\mathcal{A}$ and $\mathcal{B}$ be unital algebras over $\mathcal{R}$ and $\mathcal{M}$ be an $(\mathcal{A},\mathcal{B})$-bimodule which is faithful as a left $\mathcal{A}$-module and also as a right $\mathcal{B}$-module. The $\mathcal{R}$-algebra
$$\mathfrak{A}=\mathrm{{Tri}}(\mathcal{A},\mathcal{M},\mathcal{B})=\bigg\{\left[
                                                           \begin{array}{cc}
                                                             a & m \\
                                                             0 & b \\
                                                           \end{array}
                                                         \right]~~\vline~~ a\in\mathcal{A}, m\in\mathcal{M},b\in\mathcal{B}\bigg\}$$
under the usual matrix operations is called a triangular algebra (see \cite{Ch00} for details). Some classical examples of triangular algebras are upper triangular algebras, block upper triangular algebras and nest algebras. Define two natural projections $\pi_{\mathcal{A}}:\mathfrak{A}\rightarrow \mathcal{A}$ and $\pi_{\mathcal{B}}:\mathfrak{A}\rightarrow \mathcal{B}$ by
$\pi_{\mathcal{A}}\left(\left[
\begin{array}{cc}
a & m \\
0 & b \\
\end{array}
\right]\right)=a$ and $\pi_{\mathcal{B}}\left(\left[
\begin{array}{cc}
a & m \\
0 & b \\
\end{array}
\right]\right)=b$.\\
In view of \cite[Theorem 1.4.4]{Ch00}, the center of $\mathfrak{A}$ is given by
$$\mathcal{Z}(\mathfrak{A})=\bigg\{\left[
\begin{array}{cc}
a & 0 \\
0 & b \\
\end{array}
\right]~~\vline~~ a\in\mathcal{A},b\in\mathcal{B}, am=mb ~\text{for all}~  m\in\mathcal{M}\bigg\}.$$
Moreover, $\pi_{\mathcal{A}}(\mathcal{Z}(\mathfrak{A}))\subseteq \mathcal{Z}(\mathcal{A})$ and $\pi_{\mathcal{B}}(\mathcal{Z}(\mathfrak{A}))\subseteq \mathcal{Z}(\mathcal{B})$, and there exists a unique algebra isomorphism $\eta:\pi_{\mathcal{A}}(\mathcal{Z}(\mathfrak{A}))\rightarrow \pi_{\mathcal{B}}(\mathcal{Z}(\mathfrak{A}))$ such that $am=m\eta(a)$ for all $m\in \mathcal{M}$ and $a\in\pi_{\mathcal{A}}(\mathcal{Z}(\mathfrak{A})).$
\par Let $1_{\mathcal{A}}$ and $1_{\mathcal{B}}$ be the identity elements of $\mathcal{A}$ and $\mathcal{B}$, respectively. Set
$e=\left[
                                                           \begin{array}{cc}
                                                             1_{\mathcal{A}} & 0 \\
                                                             0 & 0 \\
                                                           \end{array}
                                                         \right]$
                                                         and $f=I-e=\left[
                                                           \begin{array}{cc}
                                                             0 & 0 \\
                                                             0 & 1_{\mathcal{B}} \\
                                                           \end{array}
                                                         \right]$.
Then $\mathfrak{A}$ can be written as $\mathfrak{A}=e\mathfrak{A}e\oplus e\mathfrak{A}f\oplus f\mathfrak{A}f$, where $e\mathfrak{A}e$ is a subalgebra of $\mathfrak{A}$ isomorphic to $\mathcal{A}$, $f\mathfrak{A}f$ is a subalgebra of $\mathfrak{A}$ isomorphic to $\mathcal{B}$ and $e\mathfrak{A}f$ is a $(e\mathfrak{A}e,f\mathfrak{A}f)$-bimodule isomorphic to the bimodule $\mathcal{M}$.
\par Throughout the paper, we shall use the following notations: For any ring (algebra) $\mathcal{A}$, let $Q_{\ell}(\mathcal{A})$ (resp. $Q_r(\mathcal{A})$) denotes the maximal left (resp. right) ring of quotients of $\mathcal{A}.$ The center of $Q_{\ell}(\mathcal{A})$ (resp. $Q_r(\mathcal{A})$) is denoted by $C_{\ell}(\mathcal{A})$ ( resp. $C_r(\mathcal{A})$).

We now state some known results which will be used in the proof of the main result of this article.

\begin{lem}\cite[Propositions 2.1 $\&$ 2.2]{W19}\label{lem2.1}
Let $\mathfrak{A}=\mathrm{{Tri}}(\mathcal{A}, \mathcal{M},\mathcal{B})$ be a triangular algebra. Then, the following assertions hold:
\begin{itemize}
\item [$(i)$] $\mathcal{Z}(\mathfrak{A})=\{c\in e\mathfrak{A}e+f\mathfrak{A}f~~~\vline~~~ cexf=exfc$ for all x $\in \mathfrak{A}\}$,
\item [$(ii)$] $C_{\ell}(\mathfrak{A})=\{q\in eQ_{\ell}(\mathfrak{A})e+fQ_{\ell}(\mathfrak{A})f~~~\vline~~~ qexf=exfq$ for all x $\in \mathfrak{A}\}$,
\item [$(iii)$] $C_r(\mathfrak{A})=\{q\in eQ_r(\mathfrak{A})e+fQ_r(\mathfrak{A})f~~~\vline~~~ qexf=exfq$ for all x $\in \mathfrak{A}\}$,
\item [$(iv)$] $\mathcal{Z}(e\mathfrak{A}e)\subseteq C_{\ell}(\mathfrak{A})e$,
\item [$(v)$] $\mathcal{Z}(f\mathfrak{A}f)\subseteq C_r(\mathfrak{A})f$,
\item [$(vi)$] there exists a unique algebra isomorphism $\tau_{\ell}:C_{\ell}(\mathfrak{A})e\rightarrow C_{\ell}(\mathfrak{A})f$ such that $\lambda e.exf=exf.\tau_{\ell}(\lambda e)$ for all $x \in \mathfrak{A}, \lambda \in C_{\ell}(\mathfrak{A})$. Moreover, $\tau_{\ell}(\mathcal{Z}(\mathfrak{A})e)=\mathcal{Z}(\mathfrak{A})f$,
\item [$(vii)$] there exists a unique algebra isomorphism $\tau_r:C_r(\mathfrak{A})e\rightarrow C_r(\mathfrak{A})f$ such that $\lambda e.exf=exf.\tau_r(\lambda e)$ for all $x \in \mathfrak{A}, \lambda \in C_r(\mathfrak{A})$. Moreover, $\tau_r(\mathcal{Z}(\mathfrak{A})e)=\mathcal{Z}(\mathfrak{A})f$.
    \end{itemize}
\end{lem}

\begin{lem}\cite[Theorem 2.1]{W19}\label{lem2.2}
Let $\mathfrak{A}=\mathrm{{Tri}}(\mathcal{A}, \mathcal{M},\mathcal{B})$ be a triangular algebra. Then $\mathfrak{A^0}=\mathrm{{Tri}}(\mathcal{A}{\tau_r}^{-1}(\mathcal{Z}(\mathcal{B})),\mathcal{M}, \mathcal{B}{\tau_{\ell}}^{-1}(\mathcal{Z}(\mathcal{A})))$
is a triangular algebra such that $\mathfrak{A}$ is a subalgebra of $\mathfrak{A^0}$ having the same identity.
\end{lem}

\section{Lie Higher Derivations of Arbitrary Triangular Algebras}

\par
Many authors studied Lie higher derivations on different rings and algebras.\linebreak Motivated by the work of Li and Shen \cite{LS12} and Moafian and Ebrahimi Vishki \cite{MEv16}, in this section we shall characterize Lie higher derivations on arbitrary triangular algebras. In order to obtain our main result, we begin with the following result.

\begin{lem}\label{lem3.1}
Let $\mathfrak{A}=\mathrm{{Tri}}(\mathcal{A},\mathcal{M},\mathcal{B})$ and $\mathfrak{A}_{1}=\mathrm{{Tri}}(\mathcal{A}_{1},\mathcal{M},\mathcal{B}_{1})$ be triangular algebras such that $\mathfrak{A}$ is a subalgebra of $\mathfrak{A}_{1}$ both having the same identity. Moreover, $\mathcal{A}\subseteq\mathcal{A}_{1}$ and $\mathcal{B}\subseteq\mathcal{B}_{1}$. A sequence $\mathcal{D}=\{\Delta_n\}_{n\in\mathbb{N}}$ of linear mappings  $\Delta_n: \mathfrak{A}\rightarrow \mathfrak{A}_{1}$ is a higher derivation if and only if for each  $n\in\mathbb{N},$ $\Delta_n$ can be written as

\begin{eqnarray}\label{eqn3.1}
  \Delta_n\left(\left[
\begin{array}{cc}
a & m \\
0 & b \\
\end{array}
\right]\right) &=& \left[
\begin{array}{cc}
f_n(a) & \sum\limits_{i+j=n,~i\neq n}(f_i(a)m_j-m_jg_i(b))+h_n(m) \\
0 & g_n(b) \\
\end{array}
\right],
\end{eqnarray}
where $\{m_j\}_{j\in\mathbb{N}} \subseteq \mathcal{M},$  $f_n:\mathcal{A}\rightarrow\mathcal{A}_{1},$ $h_n:\mathcal{M}\rightarrow\mathcal{M},$ $g_n:\mathcal{B}\rightarrow\mathcal{B}_{1}$ are linear mappings satisfying
\begin{enumerate}
  \item[$(i)$] $\{f_n\}_{n\in\mathbb{N}}$ is a higher derivation, $h_n(am)=\sum\limits_{i+j=n} f_i(a)h_j(m)$ for all $a\in\mathcal{A},$ $m\in\mathcal{M};$
  \item[$(ii)$] $\{g_n\}_{n\in\mathbb{N}}$ is a higher derivation, $h_n(mb)=\sum\limits_{i+j=n} h_i(m)g_j(b)$ for all $m\in\mathcal{M},$ $b\in\mathcal{B}.$
\end{enumerate}
\end{lem}

\begin{proof}
The proof of ``if" part is straightforward. We only need to prove ``only if" part. Assume that $\mathcal{D}=\{\Delta_n\}_{n\in\mathbb{N}}$ is a higher derivation. We shall use an induction method for $n.$ For $n=1,$ $\Delta_1:\mathfrak{A}\rightarrow \mathfrak{A}_{1}$ is a derivation and the result follows from \cite[Lemma 3.1]{W19}. Let us assume that the result holds for all integers $r<n.$  Suppose that $\Delta_n$ has the following form
  $$\Delta_n\left(\left[
\begin{array}{cc}
a & m \\
0 & b \\
\end{array}
\right]\right)=\left[
\begin{array}{cc}
f_n(a)+f^\prime_n(m)+f^{\prime\prime}_n(b) & h^\prime_n(a)+h_n(m)+h^{\prime\prime}_n(b) \\
0 & g^\prime_n(a)+g^{\prime\prime}_n(m)+g_n(b) \\
\end{array}
\right],$$
where $f_n,f^\prime_n,f^{\prime\prime}_n$ are linear mappings from $\mathcal{A},\mathcal{M},\mathcal{B}$ to $\mathcal{A}_1,$ respectively;
  $h^\prime_n,h_n,h^{\prime\prime}_n$ are linear mappings from $\mathcal{A},\mathcal{M},\mathcal{B}$ to $\mathcal{M},$ respectively;
  $g^\prime_n,g^{\prime\prime}_n,g_n$ are linear mappings from $\mathcal{A},\mathcal{M},\mathcal{B}$ to $\mathcal{B}_1,$ respectively.

Since $\mathcal{D}=\{\Delta_n\}_{n\in\mathbb{N}}$ is a higher derivation, we have
\begin{equation}\label{eqn3.2}
  \Delta_n(xy)=\sum\limits_{i+j=n}\Delta_i(x)\Delta_j(y)~~~ \mbox{for all}~~~ x,y\in\mathfrak{A}.
\end{equation}

\par Taking $x=\left[\begin{array}{cc}
a & 0 \\
0 & 0 \\
\end{array}\right]$ and $y=\left[\begin{array}{cc}
a^\prime & 0 \\
0 & 0 \\
\end{array}\right]$ in $(\ref{eqn3.2})$ yields
\begin{eqnarray*}
\lefteqn{  \left[
\begin{array}{cc}
f_n(aa^\prime) & h^\prime_n(aa^\prime) \\
0 & g^\prime_n(aa^\prime) \\
\end{array}
\right]}\\
&=& \Delta_n\left(\left[\begin{array}{cc}
aa^\prime & 0 \\
0 & 0 \\
\end{array}\right]\right)\\
&=& \Delta_n\left(\left[\begin{array}{cc}
a & 0 \\
0 & 0 \\
\end{array}\right]\left[\begin{array}{cc}
a^\prime & 0 \\
0 & 0 \\
\end{array}\right]\right)\\
&=&\left[
\begin{array}{cc}
f_n(a) & h^\prime_n(a) \\
0 & g^\prime_n(a) \\
\end{array}
\right]\left[\begin{array}{cc}
a^\prime & 0 \\
0 & 0 \\
\end{array}\right]+\left[\begin{array}{cc}
a & 0 \\
0 & 0 \\
\end{array}\right]\left[
\begin{array}{cc}
f_n(a^\prime) & h^\prime_n(a^\prime) \\
0 & g^\prime_n(a^\prime) \\
\end{array}
\right]\\
&&+\sum\limits_{i+j=n~i,j\neq n}\left[
\begin{array}{cc}
f_i(a) & \sum\limits_{p+q=i,~p\neq i}f_p(a)m_q \\
0 & 0 \\
\end{array}
\right]\left[
\begin{array}{cc}
f_j(a^\prime) & \sum\limits_{p+q=j,~p\neq j}f_p(a^\prime)m_q \\
0 & 0 \\
\end{array}
\right]\\
&=&\left[
\begin{array}{cc}
\sum\limits_{i+j=n}f_i(a)f_j(a^\prime) & ah^\prime_n(a^\prime)+\sum\limits_{i+j=n~i,j\neq n}f_i(a)h^\prime_j(a^\prime) \\
0 & 0 \\
\end{array}
\right]\\
&=&\left[
\begin{array}{cc}
\sum\limits_{i+j=n}f_i(a)f_j(a^\prime) & \sum\limits_{i+j=n~i\neq n}f_i(a)h^\prime_j(a^\prime) \\
0 & 0 \\
\end{array}
\right].
\end{eqnarray*}
Thus, $f_n(aa^\prime)=\sum\limits_{i+j=n}f_i(a)f_j(a^\prime),$ $h^\prime_n(aa^\prime)=\sum\limits_{i+j=n~i\neq n}f_i(a)h^\prime_j(a^\prime)$ and $g^\prime_n(aa^\prime)=0$ for all $a,a^\prime\in\mathcal{A}.$ Putting $a^\prime=1,$ we get $h^\prime_n(a)=\sum\limits_{i+j=n~i\neq n}f_i(a)m_j,$ where $m_j=h^\prime_j(1)$ and $g^\prime_n(a)=0$ for all $a\in\mathcal{A}.$ Repeating the same computation and choosing $x=\left[\begin{array}{cc}
0 & 0 \\
0 & b \\
\end{array}\right]$ and $y=\left[\begin{array}{cc}
0 & 0 \\
0 & b^\prime \\
\end{array}\right]$ in $(\ref{eqn3.2}),$ we get $f^{\prime\prime}_n(bb^\prime)=0,$ $h^{\prime\prime}_n(bb^\prime)=-\sum\limits_{i+j=n~i\neq n}h^\prime_i(b)g_j(b^\prime)$ and $g_n(bb^\prime)=\sum\limits_{i+j=n}g_i(b)g_j(b^\prime)$ for all $b,b^\prime\in\mathcal{B}.$ Putting $b=1,$ we get $f^{\prime\prime}_n(b^\prime)=0$ and $h^{\prime\prime}_n(b^\prime)=-\sum\limits_{i+j=n~i\neq n}m_jg_i(b)$ for all $b^\prime\in\mathcal{B}.$

\par If we choose $x=\left[\begin{array}{cc}
a & 0 \\
0 & 0 \\
\end{array}\right]$ and $y=\left[\begin{array}{cc}
0 & m \\
0 & 0 \\
\end{array}\right]$ in $(\ref{eqn3.2}),$ then we arrive at
\begin{eqnarray*}
 \left[
\begin{array}{cc}
f^\prime_n(am) & h_n(am) \\
0 & g^{\prime\prime}_n(am) \\
\end{array}
\right] &=& \Delta_n\left(\left[\begin{array}{cc}
0 & am \\
0 & 0 \\
\end{array}\right]\right)\\
&=& \Delta_n\left(\left[\begin{array}{cc}
a & 0 \\
0 & 0 \\
\end{array}\right]\left[\begin{array}{cc}
0 & m \\
0 & 0 \\
\end{array}\right]\right)\\
\end{eqnarray*}

\begin{eqnarray*}
&=&\left[
\begin{array}{cc}
f_n(a) & h^\prime_n(a) \\
0 & 0 \\
\end{array}
\right]\left[\begin{array}{cc}
0 & m \\
0 & 0 \\
\end{array}\right]+\left[\begin{array}{cc}
a & 0 \\
0 & 0 \\
\end{array}\right]\left[
\begin{array}{cc}
f^\prime_n(m) & h_n(m) \\
0 & g^{\prime\prime}_n(m) \\
\end{array}
\right]\\
&&+\sum\limits_{i+j=n~i,j\neq n}\left[
\begin{array}{cc}
f_i(a) & \sum\limits_{p+q=i~p\neq i}f_p(a)m_q \\
0 & 0 \\
\end{array}
\right]\left[
\begin{array}{cc}
0 & h_j(m) \\
0 & 0 \\
\end{array}
\right]\\
&=&\left[
\begin{array}{cc}
0 & f_n(a)m \\
0 & 0 \\
\end{array}
\right]+\left[
\begin{array}{cc}
af^\prime_n(m) & h_n(m) \\
0 & 0 \\
\end{array}
\right]+\sum\limits_{i+j=n~i,j\neq n}\left[
\begin{array}{cc}
0 & f_i(a)h_j(m) \\
0 & 0 \\
\end{array}
\right]\\
&=& \left[
\begin{array}{cc}
af^\prime_n(m) & \sum\limits_{i+j=n}f_i(a)h_j(m) \\
0 & 0 \\
\end{array}
\right]
\end{eqnarray*}
It follows that $h_n(am)=\sum\limits_{i+j=n}f_i(a)h_j(m)$ and $g^{\prime\prime}_n(am)=0$ for all $a\in\mathcal{A},$ $m\in\mathcal{M}.$ Putting $a=1,$ we get  $g^{\prime\prime}_n(m)=0$ for all $m\in\mathcal{M}.$ Similarly, consider $x=\left[\begin{array}{cc}
0 & m \\
0 & 0 \\
\end{array}\right]$ and $y=\left[\begin{array}{cc}
0 & 0 \\
0 & b \\
\end{array}\right]$ in $(\ref{eqn3.2})$ to obtain $f^\prime_n(mb)=0$ and $h_n(mb)=\sum\limits_{i+j=n}h_i(m)g_j(b)$ for all $m\in\mathcal{M},$ $b\in\mathcal{B}.$ Taking $b=1,$ we get $f^\prime_n(m)=0$ for all $m\in\mathcal{M}.$
\end{proof}

The main result of this article states as follows:

\begin{thm}\label{thm1}
Let $\mathfrak{A}=\mathrm{{Tri}}(\mathcal{A},\mathcal{M},\mathcal{B})$ be a triangular algebra and $\mathcal{L}=\{L_n\}_{n\in\mathbb{N}}$ be a Lie higher derivation on $\mathfrak{A}$. Then, there exists a triangular algebra $\mathfrak{A}^{0}$ such that $\mathfrak{A}$ is a subalgebra of $\mathfrak{A}^{0}$ having the same identity and $L_n$ can be written as $L_n= \Delta_n +\chi_n$ for each $n\in\mathbb{N},$  where $\{\Delta_n\}_{n\in\mathbb{N}},$ $\Delta_n:\mathfrak{A}\rightarrow\mathfrak{A}^{0}$ is a higher derivation and $\{\chi_n\}_{n\in\mathbb{N}}$ is a sequence of linear mappings $\chi_n:\mathfrak{A}\rightarrow\mathcal{Z}(\mathfrak{A}^{0})$ such that $\chi_n([x,y])=0$ for all $x,y\in \mathfrak{A}$ and for each $n\in\mathbb{N}.$
\end{thm}
\begin{proof}
Consider $\mathfrak{A}^{0}=\mathrm{{Tri}}(\mathcal{A}{\tau_r}^{-1}(\mathcal{Z}(\mathcal{B})),\mathcal{M},
\mathcal{B}{\tau_{\ell}}^{-1}(\mathcal{Z}(\mathcal{A}))).$ Then, in view of Lemma \ref{lem2.2}, $\mathfrak{A}^{0}$ is a triangular algebra such that $\mathfrak{A}$ is a subalgebra of $\mathfrak{A}^{0}$ having the same identity.

\par In order to obtain this theorem, we proceed by induction on $n.$ By the definition of Lie higher derivation, $L_1:\mathfrak{A}\rightarrow \mathfrak{A}$ is a Lie derivation and hence it follows from \cite[Theorem 3.1]{W19} that $L_1=\Delta_1+\chi_1,$ where $\Delta_1:\mathfrak{A}\rightarrow \mathfrak{A}^{0}$ is a derivation and $\chi_1:\mathfrak{A}\rightarrow \mathcal{Z}(\mathfrak{A}^{0})$ is a linear mapping such that $\chi_1([x,y])=0$ for all $x,y\in\mathfrak{A}.$

\par Let us assume that $L_r=\Delta_r+\chi_r$ for all integers $r<n,$ where $\Delta_r:\mathfrak{A}\rightarrow\mathfrak{A}^{0},$ $\chi_r:\mathfrak{A}\rightarrow \mathcal{Z}(\mathfrak{A}^{0})$ are linear mappings such that $\Delta_r(xy)=\sum\limits_{i+j=r} \Delta_i(x)\Delta_j(y)$ and $\chi_r([x,y])=0$ for all $x,y\in\mathfrak{A}.$ This implies that $\{\Delta_i\}_{i=0}^{r}$, $\Delta_i:\mathfrak{A}\rightarrow\mathfrak{A}^{0}$ is a higher derivation and hence, in view of Lemma \ref{lem3.1}, $\Delta_r$ can be written as

\begin{eqnarray*}
  \Delta_r\left(\left[
\begin{array}{cc}
a & m \\
0 & b \\
\end{array}
\right]\right) &=& \left[
\begin{array}{cc}
d_r(a) & \sum\limits_{i+j=r,~i\neq r}(d_i(a)m_j-m_jd^\prime_i(b))+h_r(m) \\
0 & d^\prime_r(b) \\
\end{array}
\right],
\end{eqnarray*}
where $\{m_i\}_{i=0}^{r} \subseteq \mathcal{M},$  $d_r:\mathcal{A}\rightarrow\mathcal{A}_{1},$ $h_n:\mathcal{M}\rightarrow\mathcal{M},$ $d^\prime_r:\mathcal{B}\rightarrow\mathcal{B}_{1}$ are linear mappings satisfying
\begin{enumerate}
  \item[$(i)$] $\{d_i\}_{i=0}^{r}$ is a higher derivation, $h_r(am)=\sum\limits_{i+j=r} d_i(a)h_j(m)$ for all $a\in\mathcal{A},$ $m\in\mathcal{M};$
  \item[$(ii)$] $\{d^\prime_i\}_{i=0}^{r}$ is a higher derivation, $h_r(mb)=\sum\limits_{i+j=r} h_i(m)d^\prime_j(b)$ for all $m\in\mathcal{M},$ $b\in\mathcal{B}.$
\end{enumerate}

\noindent We now show that $L_n$ has the similar decomposition, that is, $L_n=\Delta_n+\chi_n,$  where $\Delta_n:\mathfrak{A}\rightarrow\mathfrak{A}^{0},$ $\chi_n:\mathfrak{A}\rightarrow \mathcal{Z}(\mathfrak{A}^{0})$ are linear mappings such that $\Delta_n(xy)=\sum\limits_{i+j=n} \Delta_i(x)\Delta_j(y)$ and $\chi_n([x,y])=0$ for all $x,y\in\mathfrak{A}.$

\par It follows from \cite[Theorem 2.1]{MEv16} that $L_n$ has the following form
\begin{eqnarray*}
{\SMALL{  L_n\left(\left[
\begin{array}{cc}
a & m \\
0 & b \\
\end{array}
\right]\right)}}&=& {\SMALL{ \left[
\begin{array}{cc}
f_n(a)+p_n(b) & \sum\limits_{i+j=n,i\neq n}((f_i(a)+p_i(b))m_j-m_j(q_i(a)+g_i(b)))+h_n(m) \\
0 & q_n(a)+g_n(b) \\
\end{array}
\right]}}
\end{eqnarray*}
where $\{m_j\}_{j\in\mathbb{N}} \subseteq \mathcal{M},$  $f_n:\mathcal{A}\rightarrow\mathcal{A},$ $p_n:\mathcal{B}\rightarrow\mathcal{Z}(\mathcal{A}),$ $h_n:\mathcal{M}\rightarrow\mathcal{M},$ $q_n:\mathcal{A}\rightarrow\mathcal{Z}(\mathcal{B}),$ $g_n:\mathcal{B}\rightarrow\mathcal{B}$ are linear mappings satisfying
\begin{enumerate}
  \item[$(i)$] $\{f_n\}_{n\in\mathbb{N}}$ is a Lie higher derivation, $h_n(am)=\sum\limits_{i+j=n} (f_i(a)h_j(m)-h_j(m)q_i(a))$ for all $a\in\mathcal{A},$ $m\in\mathcal{M};$
  \item[$(ii)$] $\{g_n\}_{n\in\mathbb{N}}$ is a Lie higher derivation, $h_n(mb)=\sum\limits_{i+j=n} (h_i(m)g_j(b)-p_j(b)h_i(m))$ for all $m\in\mathcal{M},$ $b\in\mathcal{B};$
  \item[$(iii)$]  $p_n([b,b^\prime])=0$ and $q_n([a,a^\prime])=0$ for all $a,a^\prime\in\mathcal{A},$  $b,b^\prime\in\mathcal{B}.$
\end{enumerate}

\noindent
Employing Lemma \ref{lem2.1}$(v)$ and Lemma \ref{lem2.1}$(vii),$ we have
\begin{eqnarray}\label{eqn3.3}
h_n(am)&=&\sum\limits_{i+j=n} (f_i(a)-{\tau_r}^{-1}(q_i(a)))h_j(m)~~~\mbox{for all}~~~ a \in\mathcal{A},~~ m\in\mathcal{M}.
\end{eqnarray}
Similarly, using Lemma \ref{lem3.1}$(iv)$ and Lemma \ref{lem3.1}$(vi),$ we get
\begin{eqnarray}\label{eqn3.4}
h_n(mb)&=&\sum\limits_{i+j=n} h_i(m)(g_j(b)-\tau_{\ell}(p_j(b)))~~~\mbox{for all}~~~ b \in\mathcal{B},~~ m\in\mathcal{M}.
\end{eqnarray}
\noindent
Let us define two mappings $\Delta_n, \chi_n : \mathfrak{A}\rightarrow \mathfrak{A}^{0}$ by
\begin{eqnarray*}
 {\SMALL{ \Delta_n\left(\left[
\begin{array}{cc}
a & m \\
0 & b \\
\end{array}
\right]\right)}} &=& {\tiny{\left[
\begin{array}{cc}
f_n(a)-{\tau_r}^{-1}(q_n(a)) & \sum\limits_{i+j=n,i\neq n}((f_n(a)-{\tau_r}^{-1}(q_n(a))m_j-m_j(g_n(b)-\tau_{\ell}(p_n(b)))+h_n(m) \\
0 & g_n(b)-\tau_{\ell}(p_n(b)) \\
\end{array}
\right]}}
\end{eqnarray*}
 and
\begin{eqnarray*}
  \chi_n\left(\left[
\begin{array}{cc}
a & m \\
0 & b \\
\end{array}
\right]\right) &=& \left[
\begin{array}{cc}
{\tau_r}^{-1}(q_n(a))+p_n(b) & 0 \\
0 & q_n(a)+\tau_{\ell}(p_n(b)) \\
\end{array}
\right].
\end{eqnarray*}
It is easy to see that $\Delta_n$ and $\chi_n$ are linear mappings and $L_n=\Delta_n+\chi_n.$ Set\linebreak $d_n=f_n-{\tau_r}^{-1}\circ q_n$ and ${d^\prime}_n=g_n-\tau_{\ell} \circ p_n.$ Then, one can easily see that\linebreak $d_n:\mathcal{A}\rightarrow\mathcal{A}{\tau_r}^{-1}(\mathcal{Z}(\mathcal{B}))$ and ${d^\prime}_n:\mathcal{B}\rightarrow\mathcal{B}{\tau_{\ell}}(\mathcal{Z}(\mathcal{A}))$ are linear mappings. Moreover, $\Delta_n,$ equations $(\ref{eqn3.3})$ and $(\ref{eqn3.4})$ become
\begin{eqnarray}\label{eqn3.5}
  \Delta_n\left(\left[
\begin{array}{cc}
a & m \\
0 & b \\
\end{array}
\right]\right) &=& \left[
\begin{array}{cc}
d_n(a) & \sum\limits_{i+j=n,~i\neq n}(d_i(a)m_j-m_jd^\prime_i(b))+h_n(m) \\
0 & d^\prime_n(b) \\
\end{array}
\right],
\end{eqnarray}

\begin{eqnarray}\label{eqn3.6}
h_n(am)&=&\sum\limits_{i+j=n} d_i(a)h_j(m)~~~\mbox{for all}~~~ a \in\mathcal{A},~~ m\in\mathcal{M}
\end{eqnarray}
and
\begin{eqnarray}\label{eqn3.7}
h_n(mb)&=&\sum\limits_{i+j=n} h_i(m)d^\prime_j(b)~~~\mbox{for all}~~~ b \in\mathcal{B},~~ m\in\mathcal{M},
\end{eqnarray}
respectively. We assert that $d_n(aa^\prime)=\sum\limits_{i+j=n}d_i(a)d_j(a^\prime)$ and $d^\prime_n(bb^\prime)=\sum\limits_{i+j=n} d^\prime_i(b)d^\prime_j(b^\prime)$ for all $a, a^{\prime} \in \mathcal{A},$ $b, b^{\prime} \in \mathcal{B}.$ For any $a, a^{\prime} \in \mathcal{A},$ $m \in \mathcal{M},$ it follows from equation (\ref{eqn3.6}) and the fact that $\{d_i\}_{i=0}^{r}~~(r<n)$ is a higher derivation, we obtain
\begin{eqnarray*}
  h_n(aa^{\prime}m) &=& h_n((aa^{\prime})m)\\
    &=& \sum\limits_{i+j=n} d_i(aa^{\prime})h_j(m) \\
    &=& \sum\limits_{i+j=n,1\leq j} d_i(aa^{\prime})h_j(m)+d_n(aa^{\prime})m\\
    &=& \sum\limits_{i+j=n,1\leq j} (\sum\limits_{p+q=i}d_p(a)d_q(a^{\prime}))h_j(m)+d_n(aa^{\prime})m\\
    &=& \sum\limits_{p+q+j=n,1\leq j}d_p(a)d_q(a^{\prime})h_j(m)+d_n(aa^{\prime})m.
\end{eqnarray*}
On the other hand, we get
\begin{eqnarray*}
h_n(aa^{\prime}m)&=& h_n(a(a^{\prime}m))\\
&=&\sum\limits_{i+j=n} d_i(a)h_j(a^{\prime}m)\\
&=& \sum\limits_{i+j=n} d_i(a)(\sum\limits_{p+q=j} d_p(a^{\prime})h_q(m))\\
&=& \sum\limits_{i+p+q=n}d_i(a)d_p(a^{\prime})h_q(m)\\
&=& \sum\limits_{p+q+j=n}d_p(a)d_q(a^{\prime})h_j(m)\\
&=&\sum\limits_{p+q+j=n, 1\leq j}d_p(a)d_q(a^{\prime})h_j(m)+\sum\limits_{p+q=n}d_p(a)d_q(a^{\prime})m
\end{eqnarray*}
for all $a, a^{\prime} \in \mathcal{A}$ and $m \in \mathcal{M}$. Combining the above two equations, we arrive at
$$(d_n(aa^{\prime})-\sum\limits_{p+q=n}d_p(a)d_q(a^{\prime}))m=0$$
for all $a, a^{\prime} \in \mathcal{A}$ and $m \in \mathcal{M}$. Since $\mathcal{M}$ is a faithful left $\mathcal{A}{\tau_r}^{-1}(Z(\mathcal{B}))$-module, we get
\begin{eqnarray}\label{eqn3.8}
  d_n(aa^{\prime})&=&\sum\limits_{p+q=n}d_p(a)d_q(a^{\prime})~~\mbox{for all}~~ a, a^{\prime} \in \mathcal{A}.
\end{eqnarray}

\noindent
In a similar manner, using equation (\ref{eqn3.6}) and the fact that $\{d^\prime_i\}_{i=0}^{r}~~(r<n)$ is a higher derivation, we can prove that
\begin{eqnarray}\label{eqn3.9}
  d^\prime_n(bb^{\prime})&=&\sum\limits_{p+q=n}d^\prime_p(b)d^\prime_q(b^{\prime})~~\mbox{for all}~~ b, b^{\prime} \in \mathcal{B}.
\end{eqnarray}

\noindent
It follows from Lemma \ref{lem3.1} and equations $(\ref{eqn3.5})$-$(\ref{eqn3.9})$ that $\{\Delta_i\}_{i=0}^{n}$ is a higher derivation, that is, $\Delta_n(xy)=\sum\limits_{i+j=n}\Delta_i(x)\Delta_j(y)$ for all $x,y \in\mathfrak{A}.$

\par It only remains to prove that $\chi_n(\mathfrak{A})\subseteq \mathcal{Z}(\mathfrak{A^0})$ and $\chi_n ([x,y])=0$ for all $x,y\in \mathfrak{A}.$
For any $m \in \mathcal{M},$ we have
\begin{eqnarray*}
({\tau_r}^{-1}(q_n(a))+p_n(b))m&=&{\tau_r}^{-1}(q_n(a))m+p_n(b)m\\
&=&mq_n(a)+m\tau_{\ell}(p_n(b))\\
&=&m(q_n(a)+\tau_{\ell}(p_n(b))).
\end{eqnarray*}
Hence, by Lemma \ref{lem3.1}$(i)$, we get $\chi_n(\mathfrak{A})\subseteq \mathcal{Z}(\mathfrak{A^0}).$ Since $L_n=\Delta_n+ \chi_n$ is a Lie higher derivation and $\chi_n(\mathfrak{A})\subseteq \mathcal{Z}(\mathfrak{A^0}),$  we have
\begin{eqnarray*}
\chi_n([x,y])&=&L_n([x,y])-\Delta_n([x,y])\\
&=&\sum\limits_{i+j=n}[L_i(x),L_j(y)]-\Delta_n(xy)+\Delta_n(yx)\\
&=&\sum\limits_{i+j=n}[\Delta_i(x)+\chi_i(x),\Delta_j(y)+\chi_j(y)]-\sum\limits_{i+j=n}\Delta_i(x)\Delta_j(y)\\
&&+\sum\limits_{i+j=n}\Delta_i(y)\Delta_j(x)\\
&=&\sum\limits_{i+j=n}[\Delta_i(x),\Delta_j(y)]-\sum\limits_{i+j=n}\Delta_i(x)\Delta_j(y)+\sum\limits_{i+j=n}\Delta_i(y)\Delta_j(x)\\
&=&0
\end{eqnarray*}
for all $x,y\in \mathfrak{A}.$
\end{proof}

As a consequence of Theorem \ref{thm1}, we obtain the following result:

\begin{cor}\cite[Theorem 3.1]{W19}
Let $\mathfrak{A}=\mathrm{{Tri}}(\mathcal{A},\mathcal{M},\mathcal{B})$ be a triangular algebra and $\mathcal{L}$ be a Lie derivation on $\mathfrak{A}$. Then, there exists a triangular algebra $\mathfrak{A}^{0}$ such that $\mathfrak{A}$ is a subalgebra of $\mathfrak{A}^{0}$ having the same identity and $\mathcal{L}$ can be written as $\mathcal{L}= \Delta +\chi,$  where  $\Delta:\mathfrak{A}\rightarrow\mathfrak{A}^{0}$ is a derivation and  $\chi:\mathfrak{A}\rightarrow\mathcal{Z}(\mathfrak{A}^{0})$ is a linear mapping such that $\chi([x,y])=0$ for all $x,y\in \mathfrak{A}.$
\end{cor}

\section{Topics for further research}
Let $\mathcal{A}$ be a algebra over a unital commutative ring $\mathcal{R}.$ An $\mathcal{R}$-linear mapping $\Delta:\mathcal{A}\rightarrow \mathcal{A}$ is said to be a \textit{Lie triple derivation} if
$\Delta([[x, y],z]) = [[\Delta, y],z] + [[x, \Delta(y)],z]+[[x,y],\Delta(z)]$ for all $x, y,z\in \mathcal{A}$  holds for all $x, y\in \mathcal{A}$. Let $\mathbb{N}$ be the set of nonnegative integers and $\mathcal{D}=\{d_n\}_{n\in\mathbb{N}}$ be a family of $\mathcal{R}$-linear mappings $d_n:\mathcal{A}\rightarrow\mathcal{A}$ such that $d_0=id_{\mathcal{A}},$ the identity mapping on $\mathcal{A}.$ Then $\mathcal{D}$ is said to be a \textit{Lie triple higher derivation} on $\mathcal{A}$ if
$d_n([[x,y],z])=\sum\limits_{i+j+k=n}[[d_i(x),d_j(y)],d_k(z)]$ for all $x,y,z \in \mathcal{A}$ and  $n \in \mathbb{N}.$ Note that every Lie higher derivation is a Lie triple higher derivation. However, the converse statement is not true in general. If $\mathcal{D}=\{d_n\}_{n\in \mathbb{N}}$ is a higher derivation on $\mathcal{A}$ and $\mathcal{H}=\{h_n\}_{n\in \mathbb{N}}$ is a sequence of $\mathcal{R}$-linear mappings $h_n:\mathcal{A}\rightarrow\mathcal{Z}(\mathcal{A})$ such that  $h_n([[x,y],z])=0$ for all $x,y,z\in \mathcal{A}$ and $n \in \mathbb{N},$ then $\mathcal{D}+\mathcal{H}=\{d_{n}+h_n\}_{n\in \mathbb{N}}$ is a Lie triple higher derivation on $\mathcal{A},$ which is not necessarily a Lie higher derivation.

Lie triple (higher-) derivations in different background have been studied extensively (see \cite{aaa21,aj17,aj18,XW12(1)} and references therein). In 2012, Li and Shen \cite{LS12} showed that under certain assumptions every Lie higher derivation and Lie triple derivation on a triangular algebra are proper, respectively. Moafian and Ebrahimi Vishki \cite{MEv16} given the structure of Lie higher derivations on triangular algebras by the entries of matrices and obtained a similar conclusion as shown by Li and Shen \cite{LS12}. Ashraf and Jabeen \cite{aj18} studied nonlinear Lie triple higher derivations on triangular algebras and proved that under certain conditions every nonlinear Lie triple higher derivation  on a triangular algebra is proper. Recently, Ashraf et al. \cite{aaa21} gave a description of Lie triple derivations of arbitrary triangular algebras. In view of these results and our current work, we propose the following open problems:

\begin{con}\label{con1}
Let $\mathfrak{A}=\mathrm{{Tri}}(\mathcal{A},\mathcal{M},\mathcal{B})$ be a triangular algebra and $\mathcal{L}=\{L_n\}_{n\in\mathbb{N}}$ be a sequence of $\mathcal{R}$-linear mappings on $\mathfrak{A}.$ Then $\mathcal{L}$ is a Lie triple higher derivation on $\mathfrak{A}$ if and only if $L_n$ can be represented as
\begin{eqnarray*}
{\SMALL{  L_n\left(\left[
\begin{array}{cc}
a & m \\
0 & b \\
\end{array}
\right]\right)}}&=& {\SMALL{ \left[
\begin{array}{cc}
f_n(a)+p_n(b) & \sum\limits_{i+j=n,~~i\neq n}((f_i(a)+p_i(b))m_j-m_j(q_i(a)+g_i(b)))+h_n(m) \\
0 & q_n(a)+g_n(b) \\
\end{array}
\right]}}
\end{eqnarray*}
where $\{m_j\}_{j\in\mathbb{N}} \subseteq \mathcal{M},$  $f_n:\mathcal{A}\rightarrow\mathcal{A},$ $p_n:\mathcal{B}\rightarrow [\mathcal{A},\mathcal{A}]^{\prime},$ $h_n:\mathcal{M}\rightarrow\mathcal{M},$ $q_n:\mathcal{A}\rightarrow [\mathcal{B},\mathcal{B}]^{\prime},$ $g_n:\mathcal{B}\rightarrow\mathcal{B}$ are linear mappings satisfying
\begin{enumerate}
  \item[$(i)$] $\{f_n\}_{n\in\mathbb{N}}$ is a Lie triple higher derivation on $\mathcal{A}$, $h_n(am)=\sum\limits_{i+j=n} (f_i(a)h_j(m)-h_j(m)q_i(a))$ for all $a\in\mathcal{A},$ $m\in\mathcal{M};$
  \item[$(ii)$] $\{g_n\}_{n\in\mathbb{N}}$ is a Lie triple higher derivation on $\mathcal{B}$, $h_n(mb)=\sum\limits_{i+j=n} (h_i(m)g_j(b)-p_j(b)h_i(m))$ for all $m\in\mathcal{M},$ $b\in\mathcal{B};$
  \item[$(iii)$]  $[[p_n(b),a_1],a_2]=0,$ $p_n([[b_1,b_2],b_3])=0,$ $[[q_n(a),b_1],b_2]=0,$  $q_n([[a_1,a_2],a_3])=0$ for all $a,a_1,a_2,a_3 \in \mathcal{A},$ $b,b_1,b_2,b_3 \in \mathcal{B}.$
\end{enumerate}
\end{con}

 \begin{con}\label{con2}
Let $\mathfrak{A}=\mathrm{{Tri}}(\mathcal{A},\mathcal{M},\mathcal{B})$ be a triangular algebra and $\mathcal{L}=\{L_n\}_{n\in\mathbb{N}}$ be a Lie triple higher derivation on $\mathfrak{A}$. Then, there exists a triangular algebra $\mathfrak{A}^{0}$ such that $\mathfrak{A}$ is a subalgebra of $\mathfrak{A}^{0}$ having the same identity and $L_n$ can be written as $L_n= \Delta_n +\chi_n$ for each $n\in\mathbb{N},$  where $\{\Delta_n\}_{n\in\mathbb{N}},$ $\Delta_n:\mathfrak{A}\rightarrow\mathfrak{A}^{0}$ is a higher derivation and $\{\chi_n\}_{n\in\mathbb{N}}$ is a sequence of linear mappings $\chi_n:\mathfrak{A}\rightarrow\mathcal{Z}(\mathfrak{A}^{0})$ such that $\chi_n([[x,y],z])=0$ for all $x,y,z\in \mathfrak{A}$ and for each $n\in\mathbb{N}.$
\end{con}

\end{document}